%% file: pAdicIntegralsCrypto.tex
\documentclass[10pt]{amsart}

\input{SECTIONS/Settings.tex}
\input{SECTIONS/Commands.tex}

\begin{document}
	
	\begin{abstract}
We propose a bridge between oriented supersingular elliptic curves and the arithmetic topology of modular curves.  To an $\cO$-oriented supersingular curve, we attach a class in the relative homology group $H\left(X_0(N),C,\ZZ\right)$, i.e. modular symbols, compatible with the Hecke action. We then compute vectors of $\ell$-adic periods by pairing with weight-$2$ cusp forms via Coleman integration. This yields an explicit, computable map from short combinatorial homology representatives to truncated vectors in $(\ZZ/\ell^m\ZZ)^d$. Motivated by this encoding, we formulate the Modular Symbol Inversion (MSI) problem --recovering a short homology representative from its truncated $\ell$-adic period data-- and discuss its arithmetic structure, its relation to path problems on isogeny graphs and Bruhat–Tits trees, and potential applications to cryptographic constructions.
	\end{abstract}
	
\maketitle

\input{SECTIONS/Introduction.tex}

\input{SECTIONS/SupSingOrient.tex}

\input{SECTIONS/ModularSymbols.tex}

\input{SECTIONS/Representations.tex}

\input{SECTIONS/EquivalentConstructions.tex}

\input{SECTIONS/pAdicVectors.tex}

\input{SECTIONS/Example.tex}

\input{SECTIONS/ModSymbInvProb.tex}

\input{SECTIONS/CryptoConstructions.tex}

\input{SECTIONS/Conclusion.tex}

\input{SECTIONS/Bibliography.tex}
\end{document}

%% file: SECTIONS/Settings.tex
\usepackage{amsfonts}
\usepackage{amsmath}
\usepackage{amssymb}
\usepackage{amsthm}
\usepackage{array}
\usepackage[english]{babel}
\usepackage{calc}    
\usepackage{cite}
\usepackage{comment}
\usepackage{dashrule}
\usepackage{enumitem}
\usepackage{eso-pic}
\usepackage{eurosym}
\usepackage{fancyhdr}
\usepackage{float}
\usepackage{fontawesome}
\usepackage{graphicx}
\usepackage[hidelinks]{hyperref}
\usepackage{iftex}
\usepackage{ifthen}
\usepackage[utf8]{inputenc}	
\usepackage{lipsum}
\usepackage{mathtools}
\usepackage[expansion=true,protrusion=true]{microtype}
\usepackage{multicol}			
\usepackage{multirow}
\usepackage{parskip}
\usepackage{rotating}
\usepackage{setspace}
\usepackage{soul}
\usepackage{substr}
\usepackage{stmaryrd}
\usepackage[absolute]{textpos}
\usepackage{tikz}				
\usetikzlibrary{shapes,backgrounds,mindmap,trees}
\usetikzlibrary{arrows.meta, matrix, positioning,tikzmark,calc}
\usepackage{url}
\usepackage{wrapfig}
\usepackage{xcolor}
\usepackage{xifthen}
\usepackage{xparse}

\hypersetup{
	colorlinks   = true, 
	urlcolor     = black, 
	linkcolor    = black, 
	citecolor   = blue 
}

\title{From Orientations to $\ell$-adic Period Vectors}
\author{Leonardo Col\`o}
\date{}

\newtheoremstyle{theostyle}
{4pt} 
{} 
{} 
{} 
{\bf} 
{.} 
{.5em} 
{} 

\theoremstyle{theostyle}
\newtheorem{theorem}{Theorem}
\newtheorem{corollary}{Corollary}

\newtheorem{proposition}{Proposition}
\newtheorem{definition}[theorem]{Definition}
\newtheorem*{theorem*}{Theorem}
\newtheorem*{definition*}{Definition}

%% file: SECTIONS/Commands.tex
 \newcommand{\cO}{\mathcal{O}}   \newcommand{\cT}{\mathcal{T}}

\newcommand{\Cl}{\mathcal{C}\!\ell}
\newcommand{\GG}{\Gamma}

 \newcommand{\CC}{\mathbb{C}}
 
\newcommand{\FF}{\mathbb{F}} \newcommand{\HH}{\mathbb{H}}
\newcommand{\MM}{\mathbb{M}}
 \newcommand{\PP}{\mathbb{P}}
\newcommand{\QQ}{\mathbb{Q}} \newcommand{\RR}{\mathbb{R}}
\newcommand{\ZZ}{\mathbb{Z}}

\newcommand{\im}{\mathrm{Im}}

\newcommand{\GL}{\mathrm{GL}} \newcommand{\SL}{\mathrm{SL}}
 \renewcommand{\SS}{\mathrm{SS}}
\newcommand{\PGL}{\mathrm{PGL}}

\newcommand{\disc}{\mathrm{disc}}

\renewcommand{\fa}{\mathfrak{a}} \newcommand{\fb}{\mathfrak{b}} \newcommand{\fc}{\mathfrak{c}}

 \newcommand{\fA}{\mathfrak{A}}

\newcommand{\End}{\mathrm{End}}

\newcommand{\Aut}{\mathrm{Aut}}
\newcommand{\Pic}{\mathrm{Pic}}
\newcommand{\rank}{\mathrm{rank}}
\newcommand{\Stab}{\mathrm{Stab}}

%% file: SECTIONS/Introduction.tex
\section{Introduction}

Isogeny-based cryptography has emerged as one of the most promising families
of post-quantum public-key primitives, with schemes based on supersingular
elliptic curves, quaternion algebras, and class group actions.
Constructions such as CSIDH \cite{CSIDH}, OSIDH \cite{OSIDH}, SQISign \cite{SQISign}, and their variants exploit the
rich arithmetic, geometry and combinatorial structure of supersingular isogeny graphs to obtain compact keys and protocols with conjectured post-quantum
security.

In parallel, arithmetic geometers have long used modular symbols
and $\ell$-adic integration to study modular forms, modular curves, and the
arithmetic of elliptic curves. The modular-symbol formalism packages the homology
on the modular curve $X_0(N)$ in a combinatorial way, while overconvergent
modular symbols and harmonic cocycles on the Bruhat--Tits tree provide
effective algorithms for computing $\ell$-adic periods and $\ell$-adic $L$-values
without explicit projective models of the curves.

The central idea of this work is to use modular symbols and $\ell$-adic integrals
as an interface between oriented supersingular elliptic curves and
discrete $\ell$-adic data. Concretely, we propose to attach to an orientation
\[
\iota : \mathcal O \hookrightarrow \End(E)
\]
a homology class 
\[
\gamma(\iota) \in H_1\bigl(X_0(N),\{\text{cusps}\};\ZZ\bigr)
\]
via a class-group action on homology, and then to evaluate $\gamma(\iota)$
against weight-$2$ cusp forms by Coleman integration to obtain a truncated
$\ell$-adic period vector
\[
\Pi_m(\gamma(\iota))
\in (\ZZ/\ell^m)^d.
\]
This yields the algebraic--analytic pipeline
\begin{equation} \label{eq:pipeline-intro}
\iota
\;\longmapsto\; [\fa] \in \Pic(\cO)
\;\xrightarrow{\;\rho\;}
\gamma([\fa]) \in H_1(X_0(N),C;\ZZ)
\;\xrightarrow{\;\Pi_m\;}
\Pi_m(\gamma([\fa])) \in (\ZZ/\ell^m\ZZ)^d,
\end{equation}
where $C$ denotes the cusp set and $\rho$ is a homological representation of
the ideal class group $\Pic(\cO)$.

From a cryptographic perspective, this suggests a new family of hard problems
and primitives, distinct from but morally related to the supersingular isogeny
path problem and to lattice-based SIS/LWE. We highlight the following informal hardness
assumption:

\medskip
\noindent\textbf{Modular Symbol Inversion (MSI).}
\emph{Given a truncated $\ell$-adic period vector $y \in (\ZZ/\ell^m\ZZ)^d$ known to be
of the form $y=\Pi_m(\gamma^\star)$ for some ``short'' homology class
$\gamma^\star$ (with bounded path complexity), find any homology class
$\gamma$ of comparable complexity such that $\Pi_m(\gamma)=y$.}

\medskip

The MSI problem is encoded in the last part of \eqref{eq:pipeline-intro};
the supersingular/orientation and ideal-class layers simply provide one
way to sample short homology classes in a structured arithmetic way.
This new assumption is supported by the exponential
combinatorial complexity of homology paths and the absence of known
subexponential attacks.

%% file: SECTIONS/SupSingOrient.tex
\section{Supersingular elliptic curves and orientations}\label{sec:supersingular-orient}

In this section we recall basic facts about supersingular elliptic curves,
orders in imaginary quadratic fields, and orientations. We emphasize the
parametrization of oriented supersingular curves by ideal classes in an order
$\cO$, which underlies OSIDH \cite{OSIDH} and related constructions.

\subsection{Elliptic curves and isogenies}

We refer to~\cite{Silverman:ArithmeticEC} for a complete treatment. Throughout this work
we fix a field $k$ of positive characteristic $p$. When
$p>3$, an elliptic curve $E$ over $k$ is defined by a Weierstraß model
\[
E: y^2 = x^3 + Ax + B,\qquad A,B\in k,
\]
with non-vanishing discriminant $\Delta = -16(4A^3+27B^2)\neq 0$. The set
of $k$-rational points
\[
E(k)=\{(x,y)\in k^2 : y^2 = x^3 + Ax + B\}\,\cup\,\{O_E\}
\]
forms an abelian group under the usual chord--tangent law, with
$O_E$ as the neutral element.

The $j$-invariant of an elliptic curve $E$ in the above model is
\[
j(E) = 1728\cdot \frac{4A^3}{4A^3+27B^2},
\]
and two elliptic curves over $\overline{k}$ are isomorphic if and only if
they have the same $j$-invariant.

A separable isogeny between two elliptic curves defined over $k$
is a non-constant morphism of curves $\varphi: E_1 \longrightarrow E_2$
that is also a group homomorphism sending $O_{E_1}$ to $O_{E_2}$.
Its degree $\deg(\varphi)$ is its degree as a rational map. Isogenies
compose, and every non-constant isogeny of degree $n>1$ factors into a
composition of isogenies of prime degree whose product equals $n$.

When $\varphi$ has degree coprime to $p$, it is uniquely
determined by its kernel $\ker(\varphi)\subset E_1(\overline{k})$.
Conversely, every finite subgroup $G\subseteq E_1(\overline{k})$ of order
coprime to $p$ defines a separable isogeny
$\varphi_G:E_1\to E_1/G$, and $\varphi_G$ can be computed efficiently using
Vélu's formulas~\cite{Velu:Isogenies}.

\subsubsection*{Supersingular elliptic curves and their endomorphisms}

An elliptic curve
$E/k$ is supersingular if
it has no nontrivial $p$-torsion over
$\overline{\FF}_{p}$, i.e.,
$E(\overline{\FF}_{p})[p] = 0$.
More analytically, supersingular curves are characterized by the fact that their Newton polygon is a line segment of slope $1/2$, see \cite{Pries:NewtPoly}.

A fundamental theorem of Deuring \cite{Deuring} states that if $E$ is supersingular over
$\overline{\FF}_{p}$, then the $\QQ$-algebra
\[
\End^0(E) := \End(E)\otimes_\ZZ \QQ
\]
is the quaternion algebra $\mathfrak{A}_{p,\infty}$ ramified precisely at
$p$ and $\infty$. Moreover, $\End(E)$ is a maximal
order $R$ inside $\mathfrak{A}_{p,\infty}$.

\subsection{Orders and orientations}\label{Subsection:OrdersOrientations}

We define the notion of orientation on supersingular elliptic curves following \cite{OSIDH}.
Let $K$ be an imaginary quadratic field, and let $\cO_K$ be its ring of
integers. An \emph{order} in $K$ is a subring $\cO\subseteq\cO_K$ of finite
index such that $\cO$ is a free $\ZZ$-module of rank $2$. Every order has the
form
\[
\cO = \ZZ + f\cO_K
\]
for some integer $f\ge 1$, called the conductor. We will write
$\cO = \cO_f$ when we wish to emphasize the conductor.

\begin{definition}
A $K$-orientation on a supersingular elliptic curve $E/k$ is a homomorphism $\iota : K\to \End^0(E)$.
Let $\cO\subset \cO_K$ be an order in $K$. If $\iota(\cO)\subseteq\End(E)$, then $\iota$ is said to be an
$\cO$-orientation.  
We say that the orientation is primitive if
\[
\iota(\cO)= \iota(K)\cap \End(E),
\]
i.e. if $\cO$ is the largest quadratic order mapping inside $\End(E)$ via~$\iota$.
\end{definition}

We will always assume optimal orientations, and we will work up to oriented
isomorphism; two oriented supersingular elliptic curves $(E,\iota)$ and $(E',\iota')$ are
isomorphic if there exists an isomorphism $\phi:E\to E'$ such that $$\phi\circ \iota(a) = \iota'(a)\circ\phi \quad \text{for all } a\in\cO$$

A result of Onuki shows that
there exists an embedding $K\hookrightarrow \End^0(E)$ if and only if $p$ is
either inert or ramified in $K$, \cite{Onuki:OSIDH}. In this case there is a unique order
$\cO\subseteq\cO_K$ such that $\iota(\cO)=\iota(K)\cap \End(E)$; hence optimal orientations arise naturally from the arithmetic of $p$ in $K$.

The endomorphism ring $\End(E)$ carries a canonical character
\[
\rho:\End(E)\longrightarrow \overline{\FF}_p,
\]
defined by the action of endomorphisms on the one-dimensional space of
invariant differentials: for all $\alpha\in\End(E)$,
\[
\alpha^* \omega_E = \rho(\alpha)\,\omega_E.
\]
Composing any $\cO$-orientation $\iota$ with the reduction map
$\End(E)\rightarrow\overline{\FF}_p$ yields a \emph{$p$-orientation} on $\cO$.
If $p$ ramifies in $K$, then $\rho$ takes values in $\FF_p$ and is
self-conjugate; if $p$ is inert, $\rho$ and its conjugate $\bar\rho$ give
two distinct $p$-orientations, related by Frobenius, see \cite{MOSIDH}.

We denote by $\SS_\cO(p)$ the set of supersingular elliptic curves equipped
with an optimal $\cO$-orientation, up to oriented isomorphism, and by
$\SS_\cO(\rho)$ the subset determined by the $p$-orientation induced by~$\rho$.
When $p$ is inert in $\cO$, we obtain two disjoint subsets
$\SS_\cO(\rho)$ and $\SS_\cO(\bar\rho)$ exchanged by Frobenius.  
When $p$ is ramified, these coincide.

\subsection{Ideal classes and oriented supersingular curves}

Let $\cO$ be an order in $K$, and let $\Pic(\cO)$ denote its proper
ideal class group: the group of invertible proper $\cO$-ideals modulo
principal ideals. For background on ideal classes in non-maximal orders we refer to \cite{Cox:Primes,Deuring,Kohel1996}.

Let $E\in\SS_\cO(p)$ be a $\cO$-oriented curve with $\End(E)\cong R\subset\fA_{p,\infty}$ and let $\fa\subset\cO$ be a proper invertible $\cO$-ideal.
Using the embedding $\iota_0$, define the left $R$-ideal
\begin{equation}\label{eq:Ia-definition}
	I_\fa \;:=\; R\cdot \iota_0(\fa)
	\;\subseteq\; R.
\end{equation}
Equivalently, $I_\fa = \{\,x\in R \mid x \cdot \iota_0(\cO) \subseteq \iota_0(\fa)\,\}$. 
The ideal $I_\fa$ is locally principal at
every finite prime~$\ell\ne p$ and its reduced norm generates the same ideal of $\ZZ$ as the quadratic norm of $\fa$, i.e. $\mathrm{nrd}(I_\fa)\sim N(\fa)$, \cite{Kohel1996}.

The associated finite subgroup of $E_0$ is defined by
\[
E_0[I_\fa] := \bigcap_{x\in I_\fa} \ker(x),
\]
and has order $N(\fa)$.
The quotient yields an isogeny
\begin{equation}\label{eq:phi_a}
	\phi_\fa : E_0 \longrightarrow E_\fa := E_0/E_0[I_\fa].
\end{equation}
Since each element of $I_\fa$ commutes with $\iota_0(\cO)$, the
isogeny $\phi_\fa$ preserves $\cO$-orientation. Let $\widehat{\phi_\fa}$ denote the dual isogeny of $\phi_\fa$,
\begin{equation}\label{eq:orientation-induced}
	\iota_\fa(\alpha) \;:=\; \frac{1}{\deg\phi_\fa}\; \phi_\fa \circ \iota_0(\alpha) \circ \widehat{\phi_\fa}
	\qquad (\alpha\in\cO).
\end{equation}
is an optimal embedding $\cO\hookrightarrow \End(E_\fa)$, \cite{Deuring,Kohel1996}.

We call the pair $(E_\fa,\iota_\fa)$ the oriented isogeny transform
of $(E_0,\iota_0)$ by~$\fa$.

\begin{theorem}
	\label{thm:Pic-torsor}
The set $\SS^{pr}_{\cO}(\rho)$ of optimally $\cO$-oriented supersingular
	elliptic curves with $p$-orientation $\rho$ is a torsor for
	$\Pic(\cO)$.
\end{theorem}

Thus, isomorphism classes of supersingular elliptic curves with a fixed optimal
orientation are parameterized by~$\Pic(\cO)$.
\color{black}

\subsection{Oriented isogenies and the Bruhat--Tits tree}
\label{subsec:BT-tree}

The action of $\Pic(\cO)$ on optimally $\cO$-oriented supersingular elliptic
curves admits an interpretation in terms of $\ell$--adic geometry, through the
Bruhat--Tits tree associated with $\mathrm{PGL}_2(\QQ_\ell)$.
This viewpoint will later serve as a bridge between orientations and modular
symbols. For more information on Bruhat--Tits tree one can check \cite{Iezzi:BTTrees} and \cite{BT:Book}.

\subsubsection*{The Bruhat--Tits tree}

Let $\cT_\ell$ denote the infinite $(\ell+1)$-regular \cite{Serre:Trees} Bruhat--Tits tree of
$\mathrm{PGL}_2(\QQ_\ell)$.
Its vertices correspond to homothety classes of $\ZZ_\ell$-lattices in $\QQ_\ell^2$,
or equivalently, to isomorphism classes of maximal orders in the quaternion
algebra $\mathrm{Mat}_2\left(\QQ_\ell\right)$.  
Two vertices are connected by an edge precisely when the corresponding
lattices differ by index $\ell$, or equivalently, when there exists an isogeny
of degree $\ell$ between the corresponding supersingular elliptic curves.

Let $\Gamma := R^\times$, where $R\cong\End(E_0)$ is a fixed maximal order.
Then $\Gamma$ acts on $\cT_\ell$ without inversion, and the quotient $
\Gamma \backslash \cT_\ell$ 
is a finite graph, canonically identified with the $\ell$-isogeny
graph of supersingular elliptic curves, where each vertex is a curve $E$
and edges correspond to $\ell$-isogenies, \cite[\S 4.3]{Iezzi:BTTrees}.

\noindent{\bf Remark.}
Passing from $\cT_\ell$ to the quotient $\Gamma\backslash\cT_\ell$ identifies
vertices and edges that differ by the $\Gamma$-action, thereby folding the
infinite, cycle-free tree into a finite graph in which cycles may appear.
\subsubsection*{Oriented vertices}

Fix an optimal embedding $\iota_0:\cO\hookrightarrow R\cong\End(E_0)$.
For any $\cO$-oriented supersingular elliptic curve $(E,\iota)$, the image
$\iota(\cO)$ determines a copy of $\cO$ inside the endomorphism ring of~$E$.
This additional structure restricts which vertices and edges in
$\Gamma\backslash \cT_\ell$ are permissible.

\begin{definition}
	The \emph{oriented supersingular isogeny graph} is the subgraph of
	$\Gamma\backslash \cT_\ell$ consisting of vertices $\{(E,\iota)\}$ and edges
	corresponding to horizontal $\ell$-isogenies respecting the $\cO$-orientation.
\end{definition}

\subsubsection*{Ideal classes as oriented paths}

Let $\fa$ be a proper invertible $\cO$-ideal, and let
$(E_\fa,\iota_\fa)$ be the oriented curve associated to $\fa$ as in
Section~\ref{Subsection:OrdersOrientations}.
Let $I_\fa \subset R$ be the left ideal from \eqref{eq:Ia-definition}.
Then:

\begin{proposition}
	\label{prop:path-length}
	The ideal $I_\fa$ determines a unique geodesic path in $\Gamma\backslash\cT_\ell$
	starting at $(E_0,\iota_0)$ and ending at $(E_\fa,\iota_\fa)$.
	The length of the path is equal to the $\ell$--adic valuation of the ideal norm:
	\[
	\mathrm{length}(\fa)=v_\ell(N(\fa)),
	\]
	where $N(\fa)$ denotes the positive integer norm of~$\fa$.
\end{proposition}

\begin{proof}
	Reduction of $I_\fa$ at $\ell$- yields a $\ZZ_\ell$-lattice in $\QQ_\ell^2$ whose
	homothety class corresponds to a vertex of $\cT_\ell$.
	Multiplication by a local generator of $\fa$ at $\ell$- induces a sequence of
	index-$\ell$- sub-lattices, hence edges.  
	Since invertible $\cO$-ideals are locally principal at $\ell$, the path is
	well-defined, and its length is $v_\ell(N(\fa))$, matching the valuation of~$I_\fa$.
	The endpoint corresponds to the order $\End(E_\fa)$, giving the orientation
	$\iota_\fa$ by~\eqref{eq:orientation-induced}.
\end{proof}

\begin{corollary}
	Each ideal class $[\fa]\in\Pic(\cO)$ corresponds to the homotopy class of
	oriented paths in $\Gamma\backslash\cT_\ell$ starting at $(E_0,\iota_0)$.
\end{corollary}

\subsection{Horizontal and vertical isogenies; volcano picture}
\label{subsec:volcano}

Thus far we have fixed an order $\cO=\cO_f\subset K$ and studied the action of
its class group on $\cO$-oriented supersingular elliptic curves.  In practice,
one may vary the conductor $f$, and obtain a richer picture by considering
orientations by the family of orders
\[
\cO_f=\ZZ+f\cO_K, \qquad f\ge 1.
\]
An oriented curve $(E,\iota)$ then implicitly carries the information of
the order $\iota(\cO_f)=\iota(K)\cap\End(E)$, and isogenies between such
curves do not necessarily preserve $f$.

\subsubsection*{Horizontal versus vertical isogenies}

Let $(E,\iota)$ be an optimally $\cO_f$-oriented supersingular elliptic curve, and 
let $\phi:E\to E'$ be an isogeny of prime degree $\ell\neq p$. We endow $E'$ with the induced
$K$-action
\[
\iota'(\alpha)\ :=\ \frac{1}{\ell}\,\phi\circ \iota(\alpha)\circ \widehat{\phi}
\qquad(\alpha\in K),
\]
so that $\iota'$ is a ring homomorphism $K\to \End^0(E')$.
We then define the induced oriented order on $E'$ 
\[
\cO' :=\ \iota'(K)\cap \End(E')\ \subseteq\ \iota'(K).
\]
\begin{definition}
We say that $\phi$ is \emph{horizontal} if $\cO'=\iota'(\cO_f)$
equivalently, if $\iota'|_{\cO_f}$ is again an \emph{optimal} embedding
$\cO_f\hookrightarrow \End(E')$.  In this case the conductor is preserved.
We call $\phi$ \emph{vertical} if $\cO'\neq \iota'(\cO_f)$.
Equivalently, the induced order on $E'$ is a different quadratic order
$\cO'=\cO_{f'}\subset K$ with $f'\neq f$.
In this case the conductor changes. 
\end{definition}

Horizontal $\ell$-isogenies are precisely those compatible with the given
$\cO_f$-action; when $\ell\nmid f$ they are parametrized by proper invertible
$\cO_f$-ideals of norm $\ell$.  Vertical $\ell$-isogenies occur exactly when the
kernel has local constraints at primes dividing the conductor (in particular
at $\ell\mid f$), and they move between different conductors.

\subsubsection*{Volcano structure}

When $\ell$ is a prime dividing the conductor $f$, the $\ell$-primary part of
$\End(E)$ may change under an $\ell$-isogeny.  More precisely, if 
$\ell\mid f$ and $\phi$ is an $\ell$-isogeny, then:
\[
\End(E')\cap\iota(K) \;\in\; \{\cO_f,\cO_{f/\ell},\cO_{\ell f}\}.
\]
Hence the vertices corresponding to oriented curves with endomorphism ring
$\cO_f$ form a ``horizontal layer,'' while $\ell$-isogenies may climb upward
toward smaller conductor (larger order) or descend toward larger conductor.
This has the typical shape of an isogeny \emph{volcano}, familiar from
ordinary elliptic curves, \cite{Kohel1996,Sutherland:Volcano}.

%% file: SECTIONS/ModularSymbols.tex
\section{Modular symbols and relative homology}
\label{sec:modular-symbols}

We now recall modular symbols, the relative homology of $X_0(N)$, and the
action of Hecke operators. We fix a positive integer $N\ge 1$ throughout this
section.
Standard references include \cite{DiamonShurman:ModulaFroms,Shimura:AutoFunc,Koblitz:EllipticModular,Stein:ModularForms}.

\subsection{Modular curves and modular forms}
\subsubsection*{Congruence subgroups}

Let $\HH:=\{z\in\CC:\im(z)>0\}$
be the upper half-plane.  The group $\GL_2^+(\RR)$ acts on $\HH$ by fractional
linear transformations
\[
\gamma\cdot z = \frac{az+b}{cz+d}
\qquad\text{for}\qquad
\gamma=\begin{pmatrix}a&b\\c&d\end{pmatrix}
\]
and this restricts to an action of $\SL_2(\ZZ)$. We denote $\Gamma(N)$ the kernel of the reduction map $\SL_2(\ZZ)\to\SL_2(\ZZ/N\ZZ)$ and we call it the principal congruence subgroup of level $N$. A {\it congruence subgroup} is any subgroup of $\SL_2(\ZZ)$ that contains $\Gamma(N)$ for some $N$.
In particular, we will work with
\[
\Gamma_0(N)
=\left\{
\begin{pmatrix}a&b\\c&d\end{pmatrix}\in\SL_2(\ZZ)\;:\;c\equiv 0\pmod N
\right\}.
\]
The set of cusps of $\Gamma_0(N)$ is naturally identified with
$\Gamma_0(N)\backslash\PP^1(\QQ)$, where $\PP^1(\QQ)=\QQ\cup\{\infty\}$.
See \cite[\S1.2]{DiamonShurman:ModulaFroms} or \cite[\S1.3]{Stein:ModularForms}.


\subsubsection*{The modular curve $X_0(N)$ and its cusps}

The quotient
\[
Y_0(N):=\Gamma_0(N)\backslash\HH.
\]
defines a non-compact Riemann surface or, equivalently, a smooth complex analytic
manifold, which admits a canonical compactification by adjoining finitely
many cusps:
\[
X_0(N)=Y_0(N)\sqcup C,
\qquad
C\simeq \Gamma_0(N)\backslash\PP^1(\QQ).
\]
The compact Riemann surface $X_0(N)$ has a canonical model over $\QQ$ \cite[\S7]{Milne:Modular} and is
the coarse moduli space parametrizing elliptic curves equipped with a cyclic
subgroup of order $N$ or, equivalently, a $\Gamma_0(N)$-level structure, \cite{DeligneRap:Schemas}.
We write $g=g(X_0(N))$ for its genus, $C=\{c_1,\dots,c_c\}$ for the set of cusps and $c=\#C$ for their number.

\subsubsection*{Modular forms and cusp forms of weight $2$}

A holomorphic function $f:\HH\to\CC$ is a weight-$2$ modular form for
$\Gamma_0(N)$ if
\[
f(\gamma z)\,(cz+d)^{-2}=f(z)
\qquad\text{for all }\gamma=\begin{pmatrix}a&b\\c&d\end{pmatrix}\in\Gamma_0(N),
\]
and if it is holomorphic at every cusp, in the sense of having a Fourier
expansion with no negative powers at each cusp. The space of such forms is denoted
$M_2(\Gamma_0(N))$.

A modular form $f\in M_2(\Gamma_0(N))$ is a \emph{cusp form} if its Fourier
expansion at each cusp has vanishing constant term; the subspace of cusp forms
is denoted $S_2(\Gamma_0(N))$.
For $f\in M_2(\Gamma_0(N))$ the differential
\[
\omega_f := f(z)\,dz
\]
is $\Gamma_0(N)$-invariant on $\HH$, hence giving a meromorphic differential
on $X_0(N)$ with possible poles only at cusps. There is a canonical identification
\[
S_2(\Gamma_0(N))\ \cong\ H^0\!\bigl(X_0(N),\Omega^1_{X_0(N)}\bigr),
\qquad
f\longmapsto \omega_f,
\]
and therefore $\dim_\CC S_2(\Gamma_0(N))=g(X_0(N))$, \cite[Ch. 3]{Wehler:Notes}.

\subsubsection*{Hecke operators on modular forms}

The Hecke operators arise from natural algebraic correspondences on $X_0(N)$.
For $n\ge 1$, the \emph{Hecke correspondence} $T_n$ is induced by the finite
correspondence on $Y_0(N)$ that sends a point $(E,C)$ (with $C\subset E$ cyclic
of order $N$) to the formal sum on divisors
$$(E/C',\ (C+C')/C')$$
as $C'$ ranges over cyclic subgroups of $E$ of order $n$ and $C\cap C'=\{O\}$,
and extends to a correspondence on $X_0(N)$, see \cite[\S 1.3]{Darmon:Fermat} and \cite[\S 5.3]{DiamonShurman:ModulaFroms}.
Analytically, this correspondence is represented by a double coset operator
\[
\Gamma_0(N)\,\alpha\,\Gamma_0(N),
\qquad \alpha\in M_2(\ZZ),\ \det(\alpha)=n,
\]
acting on modular forms; see \cite[\S 5.1]{DiamonShurman:ModulaFroms}.

When $(n,N)=1$, the operator $T_n$ on weight-$2$ modular forms admits the usual
explicit formula \cite[\S 6.2]{Iwaniec:Forms}
\[
T_n(f)(z) \;=\; n\!\!\!\sum_{\substack{ad=n\\0\le b<d}}\!\!\! d^{-2}\:
f\!\left(\frac{az+b}{d}\right),
\]
and the operators $\{T_n\}_{n\ge 1}$ commute and preserve $S_2(\Gamma_0(N))$.
They are normal with respect to the Petersson inner product \cite[\S 5.5]{DiamonShurman:ModulaFroms}, so one can choose
an orthonormal basis of simultaneous eigenforms.

When $\ell\mid N$, the operator $U_\ell$ is the double-coset operator
$\Gamma_0(N)\begin{psmallmatrix}1&0\\0&\ell\end{psmallmatrix}\Gamma_0(N)$.
On weight-$2$ forms $f(q)=\sum_{n\ge1}a_n q^n$ it satisfies
\[
U_\ell (f)(q)=\sum_{n\ge1} a_{\ell n}\,q^n,
\]

\subsection{Relative homology and Manin symbols}\label{sec:ManinSymbols}
Let $C\subset X_0(N)$ denote the set of cusps on $X_0(N)$. The absolute homology $H_1(X_0(N);\ZZ)$ group is generated by homology
classes of singular 1-cycles on the Riemann surface $X_0(N)$, and more generally
encodes $1$-dimensional topological features of $X_0(N)$.  The \emph{relative homology} group $H_1(X_0(N),C;\ZZ)$
enlarges this by allowing $1$-chains whose boundary lies in $C$; in other
words, we consider paths on $X_0(N)$ whose endpoints are permitted to be
cusps, and we identify two such paths if their difference is homologous to a
sum of closed loops together with paths contained entirely in the cusps, \cite[3.2]{Stein:ModularForms}.

The group $H$ has a very concrete description in terms of modular symbols, see \cite[\S 2]{Gunnels:Notes} and \cite[\S 3.3]{Stein:ModularForms}. Let
$\widetilde{\MM}_2(\GG_0(N))$ be the free abelian group on formal {\it modular symbols} $\{r\to s\}$ with
$r,s\in\PP^1(\QQ)$, modulo the relations
\begin{align*}
	\{r\to s\} + \{s\to t\} + \{t\to r\} &= 0
	&&\text{for all } r,s,t\in \PP^1(\QQ), \\
	\{r\to r\} &= 0
	&&\text{for all } r\in \PP^1(\QQ).
\end{align*}
The group $\Gamma_0(N)$ acts on $\PP^1(\QQ)$ by linear fractional
transformations, and hence on $\widetilde{\MM}_2(\GG_0(N))$ by
\[
\gamma\cdot\{r\to s\} := \{\gamma r \to \gamma s\}.
\]
Manin observed that relative homology $H_1(X_0(N),C;\ZZ)$ can be realized
as the $\Gamma_0(N)$-coinvariants of $\widetilde{\MM}_2(\GG_0(N))$, \cite{Manin:ParabPts}.

\begin{proposition}
	There is a natural isomorphism of abelian groups
	\[
	H_1(X_0(N),C;\ZZ) \cong \MM_2(\GG_0(N)) := \GG_0(N)\backslash\widetilde{\MM_2}(\GG_0(N))
	\]
	where $\MM_2(\GG_0(N))$ 
	denotes the quotient module by the $\GG_0(N)$-action.
\end{proposition}

Elements of $H_1(X_0(N),C;\ZZ)$ are thus represented by finite $\ZZ$-linear
combinations of symbols $\{r\to s\}$ modulo the above relations and the
$\Gamma_0(N)$-action. 

\subsubsection*{Rank formula}

The rank of the relative homology $H$ can be expressed in terms of the genus
and the number of cusps.

\begin{proposition}
	Let $g=g(X_0(N))$ be the genus of $X_0(N)$ and let $c=\#C$ be the number of
	cusps. Then
	\[
	\mathrm{rk}_\ZZ\; H_1(X_0(N),C;\ZZ) = 2g + (c-1).
	\]
\end{proposition}

\begin{proof}
	Consider the long exact sequence in homology associated to the pair
	$(X_0(N), C)$:
	\[
	\cdots \to H_1(X_0(N);\ZZ) \to H_1(X_0(N),C;\ZZ)
	\to H_0(C;\ZZ) \to H_0(X_0(N);\ZZ) \to 0.
	\]
	We have $H_0(X_0(N);\ZZ)\cong \ZZ$ and $H_0(C;\ZZ)\cong \ZZ^c$. The map
	$H_0(C;\ZZ)\to H_0(X_0(N);\ZZ)$ is the augmentation map
	$\ZZ^c \to \ZZ$ sending $(n_1,\dots,n_c)\mapsto \sum_i n_i$, whose kernel has
	rank $c-1$. Since $X_0(N)$ is connected, the boundary map
	$H_1(X_0(N),C;\ZZ)\to H_0(C;\ZZ)$ is surjective onto this kernel, and we obtain
	a short exact sequence
	\[
	0 \to H_1(X_0(N);\ZZ) \to H_1(X_0(N),C;\ZZ) \to \ker(\ZZ^c\to\ZZ) \to 0.
	\]
	The result follows from the fact that $H_1(X_0(N);\ZZ)$ is free of rank $2g$, and $\ker(\ZZ^c\to\ZZ)$ is free of
	rank $c-1$.
\end{proof}


\subsubsection*{Hecke action on modular symbols} The Hecke algebra $\mathbb T$ of level $\Gamma_0(N)$ is generated by the
usual Hecke operators $T_\ell$ for primes $\ell\nmid N$ and $U_\ell$ for
$\ell\mid N$. These operators act on cusp forms of weight $2$ and level
$\Gamma_0(N)$, but also on the homology $H$ via correspondences on
$X_0(N)$, see \cite[\S 2.4]{Cremona:Algorithms}. More precisely, for each $m\ge1$ the Hecke operator is induced by the Hecke
correspondence $X_0(N)\xleftarrow{\pi_1}X_0(N,m)\xrightarrow{\pi_2}X_0(N)$,
and its action on relative homology is the push--pull map
$T_m=(\pi_2)_*\circ\pi_1^*$.

This Hecke action 
	provides a mechanism by which the class group action on supersingular
	curves 
	can be transferred to
	an action on homology, as we explain in the next section.

%% file: SECTIONS/Representations.tex
\section{Class-group representations on modular-symbol homology}
\label{sec:representations}

In this section we explain how the ideal class group $\Pic(\cO)$ gives rise
to a Hecke-equivariant action on a suitable submodule of the relative homology
$H_1(X_0(pN),C;\ZZ)$, and how this allows us to associate a homology class to
an oriented supersingular elliptic curve. We present three different approaches and prove that
they agree in $H_1(X_0(pN),C;\ZZ)$.

Throughout this section we fix an imaginary quadratic field $K$ and an order $\cO\subset K$ of discriminant $\disc(\cO)=\Delta$; we also fix a prime $p$ and a supersingular elliptic curve $E_0/\overline{\FF}_{p}$ with a primitive orientation
$\iota_0:\cO\hookrightarrow\End(E_0)$.

\subsection{Construction 1: Brandt module}

Let us fix an imaginary quadratic order $\cO\subset K$ and a primitively $\cO$-oriented supersingular curve $(E_0,\iota_0)$. Also fix a level $N$ coprime to $p$. A cyclic subgroup $C\subset E_0[N]$ defines an Eichler order $R_N\subseteq\fA_{p,\infty}$ of level $N$, and $(E_0,C)$ corresponds to a distinguished
left ideal class $I_0$ in $\Cl(\mathcal R_N)$. Each invertible $\cO$-ideal
$\fa$ defines a new oriented curve $(E_\fa,\iota_\fa)$ and a new left ideal
class $I_\fa$. 

The set of left $R_N$-ideal classes is finite, and the associated
\emph{Brandt module}
\[
\mathbb B := \ZZ[\Cl(\mathcal O_B)]
\]
is a free abelian group with basis indexed by these ideal classes.  For each
prime $\ell\nmid \mathrm{disc}(\fA_{p,\infty})\,N$, the Hecke operator $T_\ell$ acts on
$\mathbb B$ via the classical Brandt matrices, encoding $\ell$-neighbor
relations between ideal classes; see \cite{Pizer:ModForms} and
\cite[\S 3.2]{Kohel:Hecke}.

At the level of the Brandt module, the ideal action can be
encoded by a $\ZZ$-linear operator
\[
T_\fa : \mathbb B \longrightarrow \mathbb B.
\]

The Jacquet--Langlands correspondence relates the Hecke module
$\mathbb B\otimes\QQ$ to a space of weight-$2$ cusp forms on $\GL_2(\QQ)$.  More
precisely, since $\fA_{p,\infty}$ has discriminant $p$ and $R_N$ has level $N$, then
$\mathbb B\otimes\QQ$ is Hecke-isomorphic to the subspace of
$S_2(\Gamma_0(pN))$ consisting of forms that are new at $p$, see \cite[\S 4/5]{Martin:Basis} and \cite{DV:Hilbert}.  

On the other hand, by the Eichler--Shimura isomorphism, the space of weight-$2$ cusp
forms on $\Gamma_0(pN)$ embeds Hecke-equivariantly into the singular homology
of the modular curve $X_0(pN)$, see \cite{Darmon:RatPts} for more details.  Passing to relative homology with respect to the
cusps yields a Hecke-stable lattice
\[
H' \subseteq H_1(X_0(pN),C;\ZZ)
\]

Under these correspondences, each operator $T_\fa$ induces an automorphism
$T_\fa'$ of $H'$. Passing to ideal classes, we obtain a homomorphism
\[
\rho: \Pic(\cO) \to \Aut_\ZZ(H'),
\qquad [\fa] \mapsto T_\fa'.
\]

We do not claim that $\rho$ is faithful in general; its kernel depends on
$(K,\cO,N)$ and the chosen component. However, $\rho$ is nontrivial, and in
generic situations its kernel is expected to be small.

\begin{definition}
	Let $H'$ and $\rho$ as above. Fix a nonzero base class $\gamma_0\in H'$. For an ideal class $[\fa]\in\Pic(\cO)$, we define the associated homology
	class
	\[
	\gamma^{(1)}([\fa]) := \rho([\fa])(\gamma_0) \in H'.
	\]
	If $(E,\iota)$ is an $\cO$-oriented supersingular elliptic curve lying in the
	$\Pic(\cO)$-orbit of a fixed base curve $(E_0,\iota_0)$, we choose an ideal
	class $[\fa]$ such that $(E,\iota)\simeq [\fa]\star(E_0,\iota_0)$ and set
	\[
	\gamma^{(1)}(E,\iota) := \gamma^{(1)}([\fa]).
	\]
\end{definition}

The homology class $\gamma^{(1)}(E,\iota)$ depends on the choice of $[\fa]$ only up
to the stabilizer of $\gamma_0$ under the representation $\rho$.

\begin{definition}
	The stabilizer of $\gamma_0$ is the subgroup
	\[
	\Stab(\gamma_0)
	:= \{[\fc]\in\Pic(\cO) : \rho([\fc])(\gamma_0)=\gamma_0\}.
	\]
\end{definition}

If $(E,\iota)\simeq[\fa]\star(E_0,\iota_0)\simeq[\fb]\star(E_0,\iota_0)$, then
$[\fb]^{-1}[\fa]$ lies in the stabilizer of $(E_0,\iota_0)$ on the
supersingular side, which is known to be finite, see \cite{Deuring} and \cite[\S~5.2]{Kohel1996}.  Provided this finite subgroup maps into $\Stab(\gamma_0)$, the
assignment $(E,\iota)\mapsto\gamma(E,\iota)$ is well defined up to the finite
ambiguity $\Stab(\gamma_0)$.

\medskip
\noindent\textbf{Remark.}
From a cryptographic perspective, exact injectivity of the map
$(E,\iota)\mapsto\gamma(E,\iota)$ is neither expected nor required.  The map is
used to sample homology classes from a structured but exponentially large
subset of $H'$, and any bounded non-injectivity arising from finite
stabilizers does not weaken the hardness assumptions underlying the inversion
problems considered in Section~\ref{sec:MSI}.
	
The role of quaternion algebras and the
Jacquet--Langlands correspondence in this section is mostly conceptual: it
provides a representation-theoretic justification for the existence and
structure of the action we use.  All constructions needed later for
	cryptographic purposes (in particular, the computation of homology classes
	and $\ell$-adic integrals) are carried out directly on the modular curve $X_0(N)$,
	without recourse to quaternionic algorithms.

%% file: SECTIONS/EquivalentConstructions.tex
\subsection{Construction 2: Heegner points and geometric geodesic cycles}
\label{subsec:constr2}

We now describe a geometric construction of a relative homology class on the
modular curve $X_0(pN)$ associated to an ideal class in $\Pic(\cO)$, using
complex multiplication and geodesic paths on the analytic modular curve.
This construction is classical and goes back to the theory of Heegner points.

Over the complex numbers, the modular curve $X_0(pN)$ admits the analytic
uniformization $X_0(pN)(\CC)\cong
\Gamma_0(pN)\backslash \HH^\ast$ where $\HH^\ast = \HH \cup \PP^1(\QQ)$.

Assume that $(pN,\Delta)=1$ 
The theory of complex multiplication associates to each proper
invertible $\cO$-ideal $\fa$ a CM elliptic curve $E_\fa/\CC$ together with
a cyclic subgroup $C_\fa\subset E_\fa$ of order $pN$, yielding a point
\[
x_\fa \;:=\; (E_\fa, C_\fa) \;\in\; X_0(pN)(\CC).
\]

The resulting set of CM points of discriminant $\Delta$ on $X_0(pN)$ is
nonempty and is canonically parametrized by $\Pic(\cO)$, with the natural
Galois and Hecke actions corresponding to the ideal class action, \cite[Ch. 3]{Darmon:RatPts}.

Fix a base CM point $x_0=x_{\cO}$ corresponding to the trivial class, and
fix a base cusp $c_\infty\in C$, e.g.\ the class of $\infty$.
For $[\fa]\in\Pic(\cO)$ choose any continuous path
\[
\eta_\fa:[0,1]\to X_0(pN)(\CC)
\qquad\text{with}\qquad
\eta_\fa(0)=x_0,\;\eta_\fa(1)=x_\fa.
\]
Then $\eta_\fa$ defines a class in the relative homology group
$H_1(X_0(pN),\{x_0,x_\fa\};\ZZ)$.
Its boundary is
\[
\partial[\eta_\fa]=[x_\fa]-[x_0]\in H_0(\{x_0,x_\fa\};\ZZ).
\]
Choose a base cusp $c_\infty\in C$ and for each CM point $x$ choose a path $\delta_x$ from $x$ to $c_\infty$.  We can define a relative
$1$-cycle in the pair $(X_0(pN),C)$ by
\[
\widetilde\gamma^{(2)}([\fa])
\;:=\;
\eta_\fa + \delta_{x_\fa}-\delta_{x_0}.
\]
and this yields a class
\[
\gamma^{(2)}([\fa]) \;:=\;
[\widetilde\gamma^{(2)}([\fa])]
\;\in\;
H_1(X_0(pN),C;\ZZ).
\]

\subsubsection*{Independence of choices}
Changing $\eta_\fa$ with fixed endpoints changes $\eta_\fa$ by an absolute
	$1$-cycle, hence changes $\gamma^{(2)}([\fa])$ by an element of
	$H_1(X_0(pN);\ZZ)$. In the same way, changing $\delta_x$ for a fixed $x$ changes $\delta_x$ by an absolute
	$1$-cycle as well.

\subsubsection*{Modular-symbol description}
Via the Manin-symbol presentation of $H_1(X_0(pN),C;\ZZ)$ (see \S \ref{sec:ManinSymbols}), the class $\gamma^{(2)}([\fa])$ may be
represented by an explicit $\ZZ$-linear combination of Manin symbols once one
chooses matrices sending $\infty$ to the cusps and sending a fixed CM
parameter $\tau_0$ to a parameter $\tau_\fa$ for $x_\fa$.

\subsection{Construction 3: Bruhat--Tits graph and harmonic cocycles}
\label{subsec:constr3}

A third, more $\ell$-adic, viewpoint comes from the Cerednik--Drinfeld uniformization \cite{Cerednik:Uniformization,Drinfeld:padic} and the theory
of harmonic cocycles on Bruhat--Tits trees.  Throughout this subsection,
$\ell$ denotes a prime dividing $pN$ such that the relevant modular or Shimura curve admits
Cerednik--Drinfeld uniformization at $\ell$.

Over $\QQ_\ell$, the curve admits a rigid-analytic uniformization as a quotient
of the Drinfeld upper half-plane $\mathcal H_\ell$ by a discrete subgroup
$\Gamma \subset \PGL_2(\QQ_\ell)$, \cite[\S~5.3]{Darmon:RatPts}:
\[
X^{\mathrm{an}} \;\simeq\; \Gamma \backslash \mathcal H_\ell .
\]
The skeleton of $\mathcal H_\ell$ is the Bruhat--Tits tree
$\mathcal T_\ell$ of $\PGL_2(\QQ_\ell)$, and the skeleton of the quotient
$X^{\mathrm{an}}$ is the finite graph
\[
\Gamma \backslash \mathcal T_\ell .
\]
A more precise statement can be found in \cite[\S~3]{Iezzi:BTTrees}, \cite[\S~2]{Franc:Masdeu} and \cite{Teitelbaum:Drienfield}.

Let $H_1(\Gamma\backslash\mathcal T_\ell;\ZZ)$ denote the first homology of the
quotient graph.  Harmonic cocycles on $\mathcal T_\ell$ with respect to $\Gamma$
form a Hecke module naturally isomorphic to spaces of weight-$2$ modular cusp forms.
Moreover, there is a canonical Hecke-equivariant map
\[
H_1(\Gamma\backslash\mathcal T_\ell;\ZZ)
\;\longrightarrow\;
H_1(X,C;\ZZ),
\]
where $X$ denotes the corresponding algebraic curve and
$C$ its set of cusps, or boundary components in the Shimura case.
This map is induced by the specialization of analytic paths to algebraic
cycles and is compatible with the Eichler--Shimura isomorphism and with
$\ell$-adic integration, see \cite[Ch. 4]{Franc:Thesis}.

Fix a base oriented object, e.g. a primitively oriented supersingular curve, and let $v_0$ denote the corresponding vertex of the
quotient graph $\Gamma\backslash\mathcal T_\ell$.
The action of the ideal class group $\Pic(\cO)$ on oriented supersingular
curves induces, via the local embedding
$\cO\otimes\ZZ_\ell \hookrightarrow M_2(\QQ_\ell)$, an action by correspondences
on vertices of $\Gamma\backslash\mathcal T_\ell$; this is discussed explicitly in
the context of oriented curves and Bruhat--Tits trees in
\cite{Iezzi:BTTrees} and \cite[\S~4]{Franc:Masdeu}.

For an ideal class $[\fa]\in\Pic(\cO)$, choose a representative path in the
quotient graph from $v_0$ to a vertex $v_{\fa}$ corresponding to the oriented
curve $(E_\fa,\iota_\fa)$.  By fixing once and for all a spanning tree of
$\Gamma\backslash\mathcal T_\ell$, we may close this path to obtain a cycle
$c_{\fa}\in H_1(\Gamma\backslash\mathcal T_\ell;\ZZ)$.

We then define
\[
\gamma^{(3)}([\fa]) \;:=\;
\mathrm{sp}(c_{\fa})
\;\in\; H_1(X,C;\ZZ),
\]
where $\mathrm{sp}$ denotes the specialization map from graph homology to
algebraic relative homology.

\noindent{\bf Remark. }Different choices of representatives, spanning trees, or base edges modify
$c_{\fa}$ by boundaries or by cycles homologous to zero in the graph.
Under the specialization map, these changes correspond to absolute cycles in
$H_1(X;\ZZ)$, which vanish in relative homology.

\subsection{Equivalence of the three constructions}

These three constructions agree in
$H_1(X_0(pN),C;\ZZ)\otimes\QQ$, and hence 
in the integral lattice $H'$ up to finite index.

\begin{proposition}
	There exists a $\QQ$-subspace
	\[
	H'_\QQ \subseteq H_1(X_0(pN),C;\QQ),
	\]
	stable under the Hecke algebra away from $pN$, such that for every
	$[\fa]\in\Pic(\cO)$ the three constructions
	\[
	\gamma^{(1)}([\fa]),\quad
	\gamma^{(2)}([\fa]),\quad
	\gamma^{(3)}([\fa])
	\]
	define the same element of $H'_\QQ$.
\end{proposition}

\begin{proof}
On the cuspidal $\mathbb T^{(pN)}$-module $H'_{\QQ}$ the Eichler--Shimura
pairing with $S_2(\Gamma_0(pN))$ is non-degenerate, so it suffices to show that $\gamma^{(i)}([\fa])$ have the
same periods against every newform occurring in $H'_{\QQ}$.
	But (1) and (2) define the same $\mathbb T^{(pN)}$-equivariant CM/Heegner class in the newform quotient via Jacquet-Langlands and Eichler--Shimura, and (3) has the same pairings by $\ell$-adic uniformization and the identification of harmonic cocycles/graph cycles with modular symbols; hence $\gamma^{(1)}([\fa])=\gamma^{(2)}([\fa])=\gamma^{(3)}([\fa])$ in $H'_\QQ$.

\end{proof}

%% file: SECTIONS/pAdicVectors.tex
\section{$\ell$-adic period vectors and Coleman Integrals}
\label{sec:padic-vectors}

We now pass from homology classes on $X_0(N)$ to $\ell$-adic vectors via
Coleman abelian integration. 

\subsection{Weight-2 cusp forms and the period pairing}

Let $S_2(\Gamma_0(N))$ denote the space of weight-$2$ cusp forms of level
$\Gamma_0(N)$ with coefficients in $\CC$ or in a $\ell$-adic field $\QQ_\ell$.
We restrict attention to the Hecke-stable subspace corresponding to the
homology submodule $H' \subseteq H_1(X_0(N),C;\ZZ)$; concretely, this amounts
to working in the span of one or more Hecke eigenforms
$f_1,\dots,f_d \in S_2(\Gamma_0(N))$ corresponding to the newform attached to
our Brandt module.

Let $f \in S_2(\Gamma_0(N))$ be a holomorphic cusp form.
Classically, the \emph{period pairing} between $f$ and a homology class
$\gamma \in H_1(X_0(N),C;\ZZ)$ is defined by the complex integral
\[
\langle f,\gamma\rangle
\;:=\;
\int_\gamma f(z)\,dz,
\]
where $\gamma$ is represented by a singular $1$-chain on the Riemann surface
$X_0(N)(\CC)\cong \Gamma_0(N)\backslash\HH^\ast$ and $f(z)\,dz$ a holomorphic differential.
This pairing is $\CC$-bilinear and, by the Eichler--Shimura isomorphism,
induces a perfect pairing between the cuspidal part of
$H_1(X_0(N);\ZZ)\otimes\CC$ and $S_2(\Gamma_0(N))\oplus \overline{S_2(\Gamma_0(N))}$, \cite[\S~8]{Stein:ModularForms}.

A $\ell$-adic analogue of this period pairing was developed by Coleman in his
foundational work on $\ell$--adic integration on curves
\cite{Coleman:Integrals}. For a smooth curve with good reduction over a
$\ell$-adic field, Coleman defined a canonical theory of path-independent
$\ell$-adic line integrals of differentials, now known as
\emph{Coleman integrals}.

Over the last two decades, a series of works by Balakrishnan, Kedlaya, and
Tuitman developed practical algorithms for numerically computing these
integrals on curves, including methods based on explicit $\ell$-adic
cohomology and Frobenius lifts; see
\cite{BalaKedla:Integrals,Bala:ExplicitColeman,Tuitman:MapP}.
More recently, Chen, Kedlaya, and Lau \cite{CKL} introduced an efficient
approach specialized to modular curves, computing Coleman integrals directly
from modular forms data together with the
$\ell$-adic analytic uniformization.
A key feature of the Chen--Kedlaya--Lau approach is that it does not
require an explicit algebraic model of the modular curve $X_0(N)$:
instead, it works directly with $q$-expansions of modular forms and the
rigid-analytic uniformization, making it particularly well suited for
large levels and cryptographic applications

In this framework, for a weight-$2$ cusp form $f$ of finite slope at $\ell$ and
a relative homology class $\gamma \in H_1(X_0(N),C;\ZZ)$, one obtains a
well-defined $\ell$-adic period
\[
\langle f,\gamma\rangle_\ell \in \QQ_\ell,
\]
This pairing is $\QQ_\ell$-linear in both arguments and compatible with Hecke
operators,.

\subsection{The period vector}
\label{ssubec:periodvector}

Let $f_1,\dots,f_d$ be a fixed collection of weight-$2$ cusp forms corresponding
to the homology submodule $H'$. 
For $\gamma\in H'$, we define the 
$\ell$-adic period vector
\[
\Pi(\gamma) :=
\bigl(\langle f_1,\gamma\rangle_\ell\dots,\langle f_d,\gamma\rangle_\ell\bigr)
\in \QQ_\ell^d.
\]

For applications, we fix a
precision parameter $m\ge 1$ and reduce modulo $\ell^m$.

\begin{definition}
	Let $p$ be a prime 
	not dividing $N$.
	For $m\ge 1$ and $\gamma\in H'$, the \emph{truncated $\ell$-adic period vector}
	of $\gamma$ is
	\[
	\Pi_m(\gamma)
	:= \bigl(\langle f_1,\gamma\rangle_\ell,\dots,
	\langle f_d,\gamma\rangle_\ell\bigr)
	\bmod \ell^m
	\in (\ZZ/\ell^m\ZZ)^d.
	\]
\end{definition}

The map $\Pi_m:H'\to(\ZZ/\ell^m\ZZ)^d$ is $\ZZ$-linear, and its image is contained
in a subgroup whose size depends on $d$, $\ell$, and $m$. If the $f_i$'s are
chosen to be linearly independent, and the integrals are
sufficiently non-degenerate modulo $\ell^m$, then one expects $\Pi_m$ to have
full rank $d$ as a homomorphism of $\ZZ_\ell$-modules restricted to $H'\otimes\ZZ_p$.

The output space has size $\ell^{md}$, and in the following framework we will require
$\ell^{md}$ to be large compared to the number of candidate homology classes
in order to avoid excessive collisions.

%% file: SECTIONS/Example.tex
\subsection{A practical work-flow}
\label{sec:example}

In practice, we pass from an $\cO$-oriented supersingular elliptic curve to a
numerical $\ell$-adic period vector by combining Construction~2 with the
Coleman-integration algorithm of Chen--Kedlaya--Lau~\cite[\S~3]{CKL}.

Throughout, we fix a prime $p$ and work with supersingular
curves over $\overline{\FF}_{p}$ primitively oriented
by an imaginary quadratic order $\cO$ of discriminant $\Delta$.
We also fix a level $N$ with $(N,p)=1$, and a prime
$\ell\nmid Np$ used for $\ell$-adic analysis and 
integration.

\subsubsection*{Oriented curves: }Primitively $\mathcal O$-oriented
supersingular elliptic curves in characteristic $p$ are classified
by ideal classes of $\mathcal O$, or equivalently by classes of primitive
positive definite binary quadratic forms of discriminant
$\Delta=\disc(\mathcal O)$.
To a quadratic
form $Q=[a,b,c]$ one associates the CM point \cite{Stange:QuadForms}
\[
\tau=\frac{-b+\sqrt{\Delta}}{2a}\in\HH,
\]
and hence a point on $X_0(N)(\CC)$ via its $j$-invariant $j(\tau)\in\overline{\QQ}$.

In practice, one computes $j(\tau)$ either numerically via the analytic
$j$-function, and then obtain an algebraic approximation, or algebraically as a root of the Hilbert class
polynomial $H_D(X)\in\ZZ[X]$.  

\subsubsection*{Ideal action}
The action of $[\fa]\in\Pic(\cO)$ sends the CM point
$\tau$ associated to $(E,\iota)$ to the CM point $\tau_\fa$ associated to
$(E_\fa,\iota_\fa)$, yielding $j(E_\fa)=j(\tau_\fa)$ with $\tau_\fa$ determined by the corresponding quadratic form.

\subsubsection*{Level structure}
To obtain a point on the modular curve $X_0(N)$, we must additionally choose a
cyclic subgroup $C\subset E$ of order $N$. 
The resulting data $(E,C)$ defines a point $P=(E,C)\in X_0(N)(\CC)$
and the ideal action transports $(E,C)$ to corresponding level structures on
$E_\fa$, producing a point $Q:=P_\fa=(E_\fa,C_\fa)\in X_0(N)(\CC)$.

\subsubsection*{Hecke neighborhoods via modular polynomial}
We have two points $P,Q\in X_0(N)$ represented analytically by points on $\HH$ modulo $\Gamma_0(N)$.
The Hecke correspondence $T_\ell$ on $X_0(N)$ sends a point $P$ to the
collection of points corresponding to cyclic $\ell$-isogenies out of the associated elliptic curve. 
We note them as $\{j(P_i)\}_i$ and $\{j(Q_i)\}_i$

\subsubsection*{Local coordinate and residue discs}
Fix the base point $P$ and define the local parameter $t := j - j(P)$.
For a neighbor $P_i$ with $j$-invariant $j(P_i)$, we have $t(P_i)=j(P_i)-j(P)\in L.$
\subsubsection*{Differentials and their $t$-expansions}
Let $\omega$ be a holomorphic differential on $X_0(N)$.  Locally at $P$, the
differential can be expressed as a power series in $t$:
\[
\omega = \left(\sum_{n\ge 0} a_n t^n\right) dt,
\qquad a_n\in L.
\]

\subsubsection*{Tiny Coleman integrals}
Given the local expression of $\omega$ as above, a Coleman primitive is
obtained by formal integration:
\[
F_\omega(t) := \int \omega
= \sum_{n\ge 0} \frac{a_n}{n+1}\, t^{n+1}.
\]
For any point  $P_i$ in the same residue disc, the tiny integral
from $P$ to $P_i$ is computed by evaluating
\[
\int_P^{P_i}\omega \;=\; F_\omega\!\bigl(t(P_i)\bigr)
\;=\;
\sum_{n\ge 0} \frac{a_n}{n+1}\,
\bigl(j(P_i)-j(P)\bigr)^{n+1}.
\]

Coleman integrals of holomorphic
differentials on $X_0(N)$ compute the same linear functionals on
$H_1(X_0(N),C;\ZZ)$ as classical modular-symbol integrals.

\subsubsection*{Hecke symmetrization and eigenvalue normalization.}
Form Hecke-symmetrized combinations and apply the normalization factor $(\ell+1-a_\ell)^{-1}$ when working with eigen-differentials
$$(\ell+1-a_\ell)\; \int_P^Q\omega \; =\; \sum_{i=1}^{\ell+1}\; \left(\int_{Q_i}^{Q}\omega\; -\; \int_{P_i}^{P}\omega\right).$$

From the point of view of earlier sections, the $\ell$-isogenous neighbors
$P_i$ of the CM point $P$ correspond to the action of prime ideals of norm $\ell$
on the underlying quadratic form or, equivalently, on the oriented elliptic curve.  

The Hecke-symmetrized sum of tiny integrals therefore reflects the
horizontal class-group action on orientations, transported through the
Jacquet--Langlands and Eichler--Shimura correspondences to homology and
differentials on $X_0(N)$.
\subsubsection*{$\ell$-adic vector}Fix a basis $\{\omega_1,\dots,\omega_d\}$ of the relevant Hecke-stable
subspace of $S_2(\Gamma_0(N))$.  Applying the above procedure to each
$\omega_j$ yields a vector of $\ell$-adic integrals
\[
\Pi(P) :=
\bigl(
\langle \omega_1,\gamma_\fa\rangle_\ell,\dots,
\langle \omega_d,\gamma_\fa\rangle_\ell
\bigr)
\;\in\; \QQ_\ell^d,
\]
where $\gamma_\fa\in H_1(X_0(N),C;\ZZ)$ is the relative homology class determined
by the CM points $P$ and $Q$.

%% file: SECTIONS/ModSymbInvProb.tex
\section{The Modular Symbol Inversion Problem}
\label{sec:MSI}

\subsection{Path-encoded homology classes}
\label{subsec:pathencoded}

We now isolate the core hardness assumption underlying our constructions: the
difficulty of recovering a short relative homology class from partial
information about its $\ell$-adic period pairings.

Let
\[
H' \subseteq H_1(X_0(N),C;\ZZ)
\]
be the Hecke-stable $\ZZ$-lattice fixed in
Section~\ref{sec:representations}. Although the definition of $H'$ is
representation-theoretic, all homology classes used in practice will arise
from explicitly described and combinatorially simple paths, see \S~\ref{sec:representations}. These paths are most naturally described in terms of the Bruhat--Tits tree associated with $\PGL_2(\QQ_\ell)$, where
vertices correspond to suitable lattices or orientations and edges correspond
to elementary isogenies. 

In concrete cryptographic applications, and in particular when computing
$\ell$-adic period integrals, we will work with the modular-symbol
realization of Construction 2 in \S\ref{subsec:constr2} and evaluate periods using the algorithms of
Chen-Kedlaya-Lau~\cite{CKL}. This avoids the need to construct an
explicit algebraic or rigid-analytic model of $X_0(N)$ and allows direct
computation of $\ell$-adic integrals associated with modular symbols.

We therefore assume that there is a distinguished finite generating set $$\mathcal S = \{\sigma_1,\dots,\sigma_r\} \subset H'$$
such that each $\sigma_i$ represents an elementary step in the
underlying combinatorial structure, e.g.\; an oriented edge in the
Bruhat--Tits graph or a basic Manin symbol. A \emph{path of length $L$} is an
expression
\[
\gamma = \sigma_{i_1} + \cdots + \sigma_{i_L},
\]
subject to local compatibility constraints ensuring that successive steps
assemble into a valid path.

We denote by $\mathcal W_L$ the set of all valid paths of length at most $L$.
Combinatorially, the cardinality of $\mathcal W_L$ grows exponentially in $L$,
with growth rate determined by the branching of the underlying graph. This
exponential growth underlies both the expressive power of the construction
and the conjectured hardness of the inversion problems below.

\subsection{Definition of MSI}

Let
\[
\Pi_m : H' \longrightarrow (\ZZ/\ell^m\ZZ)^d
\]
be the truncated $\ell$-adic period map defined in
Section~\ref{ssubec:periodvector},.


\begin{definition}[MSI relation]
	The \emph{Modular Symbol Inversion relation} $R_{\mathrm{MSI}}$ is the subset
	of $(\ZZ/\ell^m\ZZ)^d\times\mathcal W_L$ given by
	\[
	R_{\mathrm{MSI}} := \{(y,\gamma) : \gamma\in\mathcal W_L,\; y=\Pi_m(\gamma)\}.
	\]
\end{definition}

We write $(y,\gamma)\in R_{\mathrm{MSI}}$ to indicate that $\gamma$ is a valid
\emph{short homology preimage} of $y$ under $\Pi_m$.

\begin{definition}[MSI problem]
	Given an element $y\in (\ZZ/\ell^m\ZZ)^d$ that is promised to satisfy
	\[
	y = \Pi_m(\gamma^\star)
	\]
	for some (unknown) $\gamma^\star \in \mathcal W_L$, the
	\emph{Modular Symbol Inversion (MSI) problem} is to find \emph{any}
	$\gamma\in\mathcal W_L$ such that $(y,\gamma)\in R_{\mathrm{MSI}}$.
\end{definition}

Note that $\gamma^\star$ need not be unique; there may be multiple short
paths or homology classes with the same period vector. The problem is to find
any valid witness $\gamma$.

\subsection{Comparison with SIS, LWE, and isogeny-path}

Fix a $\ZZ$-basis $\{\sigma_1,\dots,\sigma_r\}$ of $H'$, where
$r=\rank_\ZZ(H')$. Any homology class $\gamma\in H'$ can be represented by a
vector $\mathbf x\in\ZZ^r$ satisfying additional relations encoding the path
constraints.

With respect to this basis, the map $\Pi_m$ is represented by a matrix
\[
A \in M_{d\times r}(\ZZ/\ell^m\ZZ)
\]
such that
\[
\Pi_m(\gamma) \equiv A\mathbf x \pmod{\ell^m}.
\]

If one ignores the combinatorial path constraint $\gamma\in\mathcal W_L$, the
MSI problem reduces to finding a short integer vector $\mathbf x$ solving the
linear congruence $A\mathbf x \equiv y \pmod{\ell^m}$,
which is formally similar to lattice problems of SIS type, \cite{Ajtai:SIS,GPV:SIS}. However, this
analogy is limited and should not be overstated;
%
in the MSI framework, the matrix $A$ is highly structured, coming from period
	pairings, 
	and admissible vectors $\mathbf x$ are
	restricted to a sparse, exponentially small subset corresponding to valid
	paths.

Similarly, MSI differs fundamentally from LWE. In LWE, one recovers a secret
vector from noisy linear samples, \cite{Regev:LWE}. In MSI, the relation is exact, but the
difficulty arises from the structured sparsity of the solution space.
As a result, known worst-case/average-case reductions for SIS or LWE do not
apply to MSI, and no polynomial-time reduction between these problems is
currently known.

The MSI problem is conceptually closer to isogeny-based path-finding problems,
such as those underlying SQISign \cite{SQISign,CGL:Hash}, as both involve searching for short
paths in exponentially large graphs.
The analogy is nonetheless imperfect. In isogeny-based problems, the graph is
the supersingular isogeny graph, vertices are elliptic curves, and edges are isogenies. In MSI, the underlying graph is implicit:
it is the combinatorial structure generating $H'$, e.g.\ a quotient of a
Bruhat--Tits tree or the modular-symbol graph, and vertices correspond to
partial homology states rather than curves.

At present, there are no known reductions between MSI and isogeny path
problems. We treat them as distinct conjecturally hard problems, sharing only
a common exponential path-search flavor.

\subsection{Heuristic hardness and parameter choices}

Heuristically, one may model $\Pi_m$ as a random linear map on the set of
short paths. Let $\mathcal W_L$ denote the set of valid paths of length at
most $L$ and suppose $\#\mathcal W_L \approx \exp(cL)$ for some branching
constant $c>0$. If $\Pi_m$ behaves like a random function from $\mathcal W_L$
to a set of size $\ell^{md}$, then the expected number of collisions among
elements of $\mathcal W_L$ is about
\[
\frac{(\#\mathcal W_L)^2}{2\ell^{md}} \approx
\frac{\exp(2cL)}{2\ell^{md}}.
\]
Choosing parameters such that
\[
\ell^{md} \gg \exp(2cL)
\]
ensures that, with overwhelming heuristic probability, $\Pi_m$ is injective
on $\mathcal W_L$.

Even if collisions occur, the MSI problem only asks for some preimage
$\gamma\in\mathcal W_L$, not for uniqueness. The best generic attacks
on MSI are brute-force or meet-in-the-middle exploration of the path space,
	with complexity exponential in $L$ or in $L/2$ with meet-in-the-middle.
Alternatively, one can rely on lattice-based attacks on the linear system $A\mathbf x=y$, followed
	by attempts to ``round'' the resulting short vectors to valid paths.
	These are also exponential in a suitable dimension, and their effectiveness
	in exploiting the path constraint is currently unknown.

Quantumly, one can expect at most generic quadratic speedups (e.g.\ Grover
search, \cite{Grover:Quantum}) on such search spaces, leading to complexities of order
$\exp(c'L)$ for some $c'<c$ but still exponential in $L$.


As noted in Section~\ref{sec:representations}, the map from orientations to homology classes may not be injective, due to the kernel
of $\rho$ and the stabilizer $\Stab(\gamma_0)$. This means that several
orientations may yield the same homology class $\gamma$, and hence the same
period vector $\Pi_m(\gamma)$. 
This non-injectivity does not weaken the MSI assumption. The secret
object in MSI is the homology class $\gamma$, not the orientation itself. Any
finite multiplicity in the orientation-to-homology map only increases the
entropy of representations and does not provide an adversary with a shortcut
for inverting $\Pi_m$.


%% file: SECTIONS/CryptoConstructions.tex
\section{Cryptographic Constructions}
\label{sec:crypto-construct}

We now sketch two basic primitives whose security can be phrased in terms
of MSI-type assumptions. We fix the following global parameters:

\begin{itemize}[leftmargin=1.7em]
	\item a prime $p$ and an order $\cO\subset K$ in an
	imaginary quadratic field $K$;
	\item a supersingular curve $E_0/\overline{\FF}_{p}$ and an
	optimal embedding $\iota_0:\cO\hookrightarrow\End(E_0)$;
	\item a level $N$ with $(N,p)=1$ and a modular curve
	$X_0(N)$ with cusps $C$;
	\item a homology submodule $H'\subseteq H_1(X_0(N),C;\ZZ)$ and a
	representation $\rho:\Pic(\cO)\to\Aut_\ZZ(H')$ as in
	Section~\ref{sec:representations};
	\item a base class $\gamma_0\in H'$ and a generating set $\mathcal S$ of
	elementary path steps, together with a path length bound $L$ defining
	$\mathcal W_L$;
	\item an analysis prime $\ell$, a precision $m$, and a set of cusp forms
	$f_1,\dots,f_d$ defining $\Pi_m:H'\to(\ZZ/\ell^m\ZZ)^d$.
\end{itemize}

\subsection{An identification protocol}

A user chooses as secret key a random short homology class
$\gamma_{\mathrm{sk}}\in\mathcal W_L$ and sets
\[
y_{\mathrm{pk}} = \Pi_m(\gamma_{\mathrm{sk}})
\]
as public key. The identification protocol proceeds as follows:

\begin{enumerate}[leftmargin=1.7em]
	\item Prover (with secret $\gamma_{\mathrm{sk}}$) commits to a random
	path $\gamma_{\mathrm{com}}\in \mathcal W_L$ and sends
	$t = \Pi_m(\gamma_{\mathrm{com}})$ to the Verifier.
	
	\item Verifier samples a random challenge
	$c\in\{0,1\}$, or more generally $c\in\ZZ_q$ for a small public modulus $q$,
	and sends $c$ to the Prover.
	
	\item Prover computes the response homology class
	\[
	\gamma_{\mathrm{resp}} := \gamma_{\mathrm{com}} + c\,\gamma_{\mathrm{sk}}
	\;\in\; H',
	\]
	using a fixed reduction procedure to ensure that
	$\gamma_{\mathrm{resp}}$ is again represented by a valid short path
	$\gamma_{\mathrm{resp}}\in\mathcal W_{L'}$, and sends
	$\gamma_{\mathrm{resp}}$ to the Verifier.
	
	\item \emph{Verification.}
Verifier checks that $\gamma_{\mathrm{resp}}\in\mathcal W_{L'}$ and that
	\[
	\Pi_m(\gamma_{\mathrm{resp}})
	\;\equiv\;
	t + c\,y_{\mathrm{pk}}
	\pmod{\ell^m}.
	\]
\end{enumerate}

Completeness follows from the homomorphism property of $\Pi_m$. Special
soundness holds in the usual sense: given two accepting transcripts with the
same commitment $t$ and two distinct challenges $c\neq c'$, one can extract
\[
\gamma_{\mathrm{sk}}
=
\frac{\gamma_{\mathrm{resp}}-\gamma_{\mathrm{resp}}'}{c-c'}
\]
as a short homology class solving the MSI problem. The hardness of producing
such a witness without knowledge of $\gamma_{\mathrm{sk}}$ is therefore
reduced to MSI. Honest-Verifier zero-knowledge follows from the fact that
commitments $t$ are distributed as images of random short paths.

\subsection{A PRF based on iterated period mappings}

One may also envision pseudorandom functions keyed by short homology classes,
in the spirit of lattice- and isogeny-based PRFs \cite{PRF:Expnader,Goldreich:PRF,NTPRFs}.

Let $\gamma_{\mathrm{sk}}\in\mathcal W_L$ be the secret key. Given an input
bitstring $x\in\{0,1\}^*$, we can interpret $x$ as a word in the generators
$\mathcal S$, yielding a short path $\gamma_x\in\mathcal W_{L_x}$. Define a
combined path $\gamma_{\mathrm{sk},x}$ using a fixed path-combination rule
such as concatenation followed by reduction to bounded length. We output
\[
F_{\gamma_{\mathrm{sk}}}(x)
:=
\mathsf{KDF}\bigl(\Pi_m(\gamma_{\mathrm{sk},x})\bigr),
\]
where $\mathsf{KDF}$ is a standard hash-based key-derivation, e.g.\ HKDF~\cite{RFC5869} or a NIST-approved PRF-based KDF~\cite{NIST800108}
used to map $(\ZZ/\ell^m\ZZ)^d$ to a uniformly distributed bitstring.

Assuming the hardness of MSI and suitable pseudorandomness properties of
$\Pi_m$ on short paths, the resulting function is expected to be
computationally indistinguishable from a random function for adversaries
without knowledge of $\gamma_{\mathrm{sk}}$. A full proof would require a
precise model of the combinatorial structure of $\mathcal W_L$ and of
potential correlations introduced by the path-combination operation.

\subsection{Security parameters and parameter selection}
\label{subsec:security-params}

We choose parameters so that recovering a short path
$\gamma\in\mathcal W_L$ from its truncated $\ell$-adic period vector
\[
\Pi_m(\gamma)\in (\ZZ/\ell^m\ZZ)^d
\]
requires exponential work.  Here $\ell$ is the prime used for $\ell$-adic
integration, $m$ is the truncation depth, and $d$ is the number of independent
period coordinates, which is typically the dimension of the chosen Hecke component.

Let $B$ be the effective branching factor of the underlying path model, e.g.\; Bruhat--Tits trees or Manin-symbol dynamics.
Heuristically $\#\mathcal W_L \approx B^L$, so generic search costs
$\Theta(B^L)$, while meet-in-the-middle costs $\Theta(B^{L/2})$.  To target
$\lambda$-bit classical security we require $B^L\gtrsim 2^\lambda$ ($B^{L/2}\gtrsim 2^\lambda$ under
generic quantum quadratic speedups).

The output space has size $\#(\ZZ/\ell^m\ZZ)^d=\ell^{md}$.  To suppress collision-
style and meet-in-the-middle attacks that exploit $\Pi_m$, we impose the separation condition
\[
\ell^{md}\gtrsim (\#\mathcal W_L)^2 \approx B^{2L},
\]
which heuristically makes $\Pi_m$ essentially injective on $\mathcal W_L$.

For efficient $\ell$-adic integration we typically take $\ell\in\{3,5\}$ and then
choose $(m,L)$ to satisfy the inequalities above.  The level $N$ governs both
the ambient lattice size and the number of available period coordinates:
\[
r=\rank_\ZZ H_1(X_0(N),C;\ZZ)=2g(X_0(N))+\#C-1,
\qquad
d\le \dim S_2(\Gamma_0(N)).
\]
Increasing $N$ tends to increase $r$ and $d$, strengthening the entropy
$\ell^{md}$ but also increasing the cost of modular-symbol and modular-form
computations.  In prototypes we may take small $N$ and compensate with larger
$(m,L)$; for security-oriented parameters we take $N$ so that $d$ is large
(e.g.\ $d\ge 128$) allowing moderate $m$ while keeping generic attacks beyond
$2^{128}$.

Finally, the characteristic prime $p$ used for supersingular
sampling is independent of $q$; in security-oriented instantiations we take
$p$ at the $\sim256$-bit scale as in OSIDH/SQISign-style
parametrizations, while $\ell\in\{3,5\}$ is reserved for efficient period
computations.

%% file: SECTIONS/Conclusion.tex
\section{Conclusion and future work}
\label{sec:conclusion}

We have proposed a new algebraic--analytic encoding of oriented supersingular
elliptic curves into modular symbols and $\ell$-adic period vectors.

On the cryptographic side, we isolated the Modular Symbol Inversion problem as a natural candidate hardness assumption: given a period vector
$y=\Pi_m(\gamma^\star)$ arising from a short homology class, find any short
homology class $\gamma$ with $\Pi_m(\gamma)=y$. MSI sits at the intersection
of lattice linear algebra and combinatorial path problems. 

We sketched how MSI can underlie identification schemes, signatures, and pseudorandom functions. These constructions are at an exploratory stage and future work consists in analyzing parameter selection,
implementations, and resistance to structural attacks.